\documentclass[a4paper,12pt]{article}
\usepackage{amsmath,amssymb}
\usepackage{amsthm}
\usepackage[dvipdfmx]{graphicx}
\usepackage{mathrsfs}
\usepackage[active]{srcltx}
\usepackage{amscd}
\usepackage{cases}
\usepackage{cite}
\usepackage{times}
 \newtheorem{theorem}{Theorem}[section]
 \newtheorem{proposition}[theorem]{Proposition}
 
 \newtheorem{lemma}[theorem]{Lemma}
 
\theoremstyle{definition}
 \newtheorem{definition}[theorem]{Definition}
 \newtheorem{remark}[theorem]{Remark}
 
\numberwithin{equation}{section}

\newcommand{\R}{\boldsymbol{R}}
\newcommand{\N}{\boldsymbol{N}}

\newcommand{\A}{\mathcal{A}}

\newcommand{\E}{\mathcal{E}}
\newcommand{\pmt}[1]{{\begin{pmatrix} #1  \end{pmatrix}}}

\newcommand{\sgn}{\operatorname{sgn}}

\newcommand{\rank}{\operatorname{rank}}

\renewcommand{\phi}{\varphi}

\renewcommand{\Gamma}{\varGamma}
\newcommand{\ep}{\varepsilon}
\newcommand{\RR}{{\mathcal R}}
\newcommand{\trans}[1]{{\vphantom{#1}}^t{\!#1}}
\usepackage[dvipsnames]{xcolor}

\begin{document}
\title{Geometry of $D_4^-$-front singularities and a Gauss-Bonnet type formula}
\author{Kentaro Saji}
\date{\today}
\maketitle

\footnote[0]{
Partly supported by the
JSPS KAKENHI Grant Numbers 25K07001 and 22KK0034.}

\begin{abstract}
We construct a form of the $D_4^-$-singularity of fronts 
in $\R^3$ which uses
coordinate transformation on the source and isometry on the target.
As an application, we compute differential geometric invariants 
near the $D_4^-$-singularity, and give
a Gauss-Bonnet type theorem for one-parameter generic fronts.
\end{abstract}

\section{Introduction}
Wave fronts and frontals in three-dimensional space constitute a class of
singular surfaces that admit well-defined normal vector even along
their singular sets.  
Over the past decades, the differential
geometry of such objects has been extensively studied, 
with a particular focus on
the local geometry of generic singularities such as cuspidal
edges and swallowtails. 
These singularities, together with their associated
invariants, play a central role in the geometry of surfaces with
singularities and have been studied from various viewpoints.
\cite{dz,hhnsuy,hhnuy,hnuy,ms,msuy,nuy,OTflat,suyfront}.
The fundamental differential geometric invariants of cuspidal edges is
introduced in \cite{suyfront}. 
When one considers generic one-parameter families of fronts, additional
bifurcation phenomena arise, including several corank-one bifurcations and
the $D_4^\pm$-type bifurcations described in \cite[Section~22.1]{agv}.
In this paper, we refer to the central singular point occurring in a
$D_4^\pm$-bifurcation simply as a \emph{$D_4^\pm$-singularity}.

The notion of an $SO(3)$-{\it normal form}, or
simply a {\it normal form\/} used
here, is a parametrization obtained by
appropriate coordinate changes in the source and isometries in the target
\cite{WE}. Normal form produces one method to investigate
higher-order geometric invariants of singular point.
For a deeper geometric understanding of singularities, higher-order
invariants are indispensable.  
Higher order invariants of the cuspidal edges are
studied in \cite{nuy}, where the
moduli of isometric deformations were also determined.
In \cite{d4geom},
a normal form for $D_4^+$-singularity are obtained and
fundamental geometric properties are studied.

A $D_4^-$-singularity appears on
the parallel surface of a surface
at ellipitic/hyperbolic umbilic point, 
and it also appears
on minimal surfaces,
this singularity is geometrically remarkable.
In this paper, we focus on the $D_4^-$-singularity.  
We construct a normal form for this singularity, compute the associated
differential geometric invariants in its neighborhood, and show a
Gauss-Bonnet type formula for fronts that may admit this class of
singular points.

The precise definition of fronts, and
the $D_4^-$-singularity is given as follows:
The unit cotangent bundle $T^*_1\R^{3}$ of $\R^{3}$ has the
canonical contact structure and can be identified with the unit
tangent bundle $T_1\R^{3}$. Let $\omega$ denote the canonical
contact form on it. A map $i:M\to T_1\R^{3}$ is said to be 
{\it
isotropic\/} if the pull-back $i^*\omega$ vanishes
identically, where $M$ is a $2$-manifold. 
If $i$ is an immersion, then the image $i(M)$ is a Legendre
submanifold, and the image of $\pi\circ i$ is called
the {\it wave front set},
where $\pi:T_1\R^{3}\to\R^{3}$ is the canonical projection and we
denote it by\/ $W(i)$. Moreover, $i$ is called the 
{\it Legendrian lift\/} of $W(i)$. 
With this framework, we define the notion of
fronts as follows: A map-germ $f:(\R^2,0) \to (\R^{3},0)$ is
called a {\it frontal\/} 
if there exists a unit vector field 
(called {\it unit normal of\/} $f$)
$\nu$ of $\R^{3}$ along $f$
such that
$L=(f,\nu):(\R^2,0)\to (T_1\R^{3},0)$ is
an isotropic map by an identification 
$T_1\R^3 = \R^3 \times S^2$, where $S^2$ is 
the unit sphere in $\R^3$ (cf. \cite{agv}, see also \cite{krsuy}).
A frontal $f$ is a {\it front\/} if the above $L$ can be taken as
an immersion.
A point $q\in (\R^2, 0)$ is a singular point if $f$ is not an
immersion at $q$.
A map-germ $f:(\R^2,0)\to(\R^3,0)$ is called the
{\it $D_4^-$-singularity\/}
if $f$ is $\mathcal{A}$-equivalent to
\begin{equation}\label{eq:f0}
f_0^-=
\left(
\dfrac{u^2}{2}-\dfrac{v^2}{2},\ 
u v,\ 
\dfrac{3u^2 v- v^3}{2}
\right)
\end{equation}
at the origin.
Here, two map-germs $f,g:(\R^2,0)\to(\R^3,0)$
is said to be ${\mathcal A}$-{\it equivalent\/}
if there exist diffeomorphism-germs
$\phi:(\R^2,0)\to(\R^2,0)$ and 
$\Phi:(\R^2,0)\to(\R^3,0)$ such that
$$
\Phi\circ g\circ \phi^{-1}=f
$$
holds.
We remark that the function
 $h=u^2v/2-v^3/3!$ has a $D_{4}^-$-singularity
at the origin, and 
the discriminant set ${\mathcal D}_H$ of a Moser family
$H=h-u y - v x + z$ of $h$
is parameterized by
$$
{\mathcal D}_H=\left\{\left(
\dfrac{u^2}{2}-\dfrac{v^2}{2},\ 
u v,\ 
\dfrac{3u^2 v- v^3}{2}
\right)
\,\Big|\,(u,v)\in\R^2\right\}.
$$
\section{Normal form}
\subsection{Preliminaries for normal form}
For a frontal $f:(\R^2,0)\to (\R^3,0)$
with a unit normal $\nu$ and for a coordinate system $(u,v)$,
the function 
\begin{equation}\label{eq:lambda}
\lambda=\det(f_u,f_v,\nu)
\end{equation}
is called the {\it signed area density function}
$($with respect to $(u,v)$ and $\nu)$.
The signed area density function is multiplied by a
non-zero function by changing coordinate system
and a unit normal vector field.

We set $\E_2=C^\infty(2,1)$ the set of function germs of two variables
at the origin, and $\langle u^2+v^2\rangle_{\E_2}$
the ideal generated by $u^2+v^2$ in $\E_2$.
\begin{definition}
A function $h:(\R^2,0)\to(\R^k,0)$ is 
{\it rotational-radial compatible} with respect to a
coordinate system $(u,v)$
if
\begin{equation}\label{eq:compati}
uh_u-vh_v\in\langle u^2+v^2\rangle_{\E_2},\quad
vh_u+uh_v\in\langle u^2+v^2\rangle_{\E_2}.
\end{equation}
\end{definition}
One can easily see that the condition
\eqref{eq:compati} is equivalent to
\begin{equation}\label{eq:compati2}
h_u=u\bar{p}+v\bar{q},\quad
h_v=-v\bar{p}+u\bar{q}
\end{equation}
for some functions $\bar{p},\bar{q}$.
If $h$ is written as 
\begin{equation}\label{eq:dset}
h(u,v)=h_1(u^2,v^2)+uh_2(u^2,v^2)+vh_3(u^2,v^2)+uvh_4(u^2,v^2),
\end{equation}
then \eqref{eq:compati} is equivalent to
\begin{equation}\label{eq:diffeq}
\begin{array}{ll}
(h_1)_u+(h_1)_v\Big|_{(u,v)=(t,-t)}&=0,\\
h_2-2v((h_2)_u+(h_2)_v)\Big|_{(u,v)=(t,-t)}&=0,\\
h_3+2v((h_3)_u+(h_3)_v)\Big|_{(u,v)=(t,-t)}&=0,\\
(h_4)_u+(h_4)_v\Big|_{(u,v)=(t,-t)}&=0.
\end{array}
\end{equation}
Although this condition is stated for a function $h$, 
it can be expressed in terms of the homogeneous components 
of the same degree as follows.
Let $n\in\N$ and let a function $h$ is written as
$h=\sum_{k=2}^n\sum_{i+j=k}
h_{ij}u^iv^j/(i!j!)$. Then the condition
\eqref{eq:diffeq} for the terms $i+j\leq 6$ are
\begin{equation}\label{eq:rotradfunc}
\begin{array}{rc}
h_{02}+h_{20}&=0,\\
(3h_{12}+h_{30},h_{03}+3h_{21})&=(0,0),\\
(-h_{04}+h_{40},h_{13}+h_{31})&=(0,0),\\
(-3h_{14}-2h_{32}+h_{50},-h_{05}+2h_{23}+3h_{41})&=(0,0),\\
(h_{06}-5h_{24}-5h_{42}+h_{60},
-h_{15}+h_{51})&=(0,0).
\end{array}
\end{equation}
Let $f:(\R^2,0)\to(\R^3,0)$ be a frontal satisfying
$df_0=0$.
A coordinate system $(u,v)$ is {\it adapted}\/ if
there exist a pair of linearly independent
vector fields $\{p,q\}$ along $f$ such that
\begin{equation}\label{eq:compati1}
f_u=up+vq,\quad f_v=-vp+uq.
\end{equation}
The condition \eqref{eq:compati1} is equivalent to
\begin{equation}\label{eq:compati3}
uf_u-vf_v=(u^2+v^2)p,\quad vf_u+uf_v=(u^2+v^2)q.
\end{equation}
The above pair $\{p,q\}$ obtained 
in the coordinate system $(u,v)$,
is called an {\it associated frame} of $f$.
The condition \eqref{eq:compati1} implies that the
each coordinate function $f_1,f_2,f_3$ of $f=(f_1,f_2,f_3)$ is 
rotational-radial compatible with respect to $(u,v)$.
On an adapted coordinate system,
the signed area density function is a non-zero multiple
of $u^2+v^2$.
However, the property that the signed area density function is 
a non-zero multiple
of $u^2+v^2$, does not imply that it is an adapted coordinate system.
For example, if $(u,v)$ and $f$ satisfy
$f_u=up+vq$ and $f_v=(-v+u a(u,v))p+(u+va(u,v))q$,
then the signed area density function is 
$(u^2+v^2)|p\times q|^2$. 

In this paper, we consider
a frontal $f$ satisfying $df_0=0$ and that
there exists an adapted coordinate system.
We see the condition for
front.
\begin{lemma}\label{lem:basisplfr}
Let $f$ be a frontal satisfying
$\rank df_0=0$.
We assume that a coordinate system $(u,v)$
is adapted, and $\{p,q\}$ the associated frame.
Then $f$ is a front at $0$ if and only if
\begin{equation}\label{eq:upvqfr}
\det\pmt{
\det(p_u,p,q)&\det(p_v,p,q)\\
\det(q_u,p,q)&\det(q_v,p,q)}(0,0)\ne0.
\end{equation}
\end{lemma}
\begin{proof}
Set $\nu=p\times q$. Then $\nu$ is a normal vector field
not necessary unit.
Thus $f$ is a front if and only if $\det(\nu,\nu_u,\nu_v)\ne0$.
By a fundamental vector calculus,
$$
\det(\nu,\nu_u,\nu_v)
=
\det(p\times q,p_u\times q+p\times q_u,
p_v\times q+p\times q_v)
=
\det\pmt{
\det(p_u,p,q)&\det(p_v,p,q)\\
\det(q_u,p,q)&\det(q_v,p,q)}
$$
holds, and this shows the assertion.
\end{proof}
For the germ $f_0^-$ in \eqref{eq:f0},
there exists a coordinate system
$(u,v)$ and $\{p,q\}$ such that
$$
(f_0^-)_u=up+vq,\quad
(f_0^-)_v=-vp+ u q
$$
holds, where
$$
p=(1,0,v),\quad q=(0,1,u).
$$
The vector $\nu=p\times q/|p\times q|$ gives a unit normal.
If $f$ is a $D_4^-$-singularity,
then there exists an adapted coordinate system as
the following lemma.
\begin{lemma}\label{lem:basispl}
Let $f$ be a $D_4^-$-singularity.
Then there exists $g$ such that
$g$ is $\RR$-equivalent to $f$ and 
there exists an adapted coordinate system.
\end{lemma}
\begin{proof}
Since $f$ is a $D_4^-$-singularity,
there exist diffeomorphisms $\phi:(\R^2,0)\to(\R^2,0)$
and $\Phi:(\R^3,0)\to(\R^3,0)$
such that
$f(x,y)=\Phi\circ f_0^\ep\circ \phi(x,y)$.
Namely,
$f\circ \phi^{-1}(u,v)=\Phi\circ f_0^\ep(u,v)$ holds.
We set
$g(u,v)=f\circ \phi^{-1}(u,v)$. Then $g$ is 
$\RR$-equivalent to $f$.
We denote by
$(X,Y,Z)$
the coordinate system on the target.
Setting
$f_0^\ep=(f_{0,1}^\ep,f_{0,2}^\ep,f_{0,3}^\ep)$,
it holds that
\begin{align*}
g_u&=\Phi_X(f_{0,1}^\ep)_u+\Phi_Y(f_{0,2}^\ep)_u+\Phi_Z(f_{0,3}^\ep)_u
=u\Phi_X+v\Phi_Y+2uv\Phi_Z,\\
g_v&=\Phi_X(f_{0,1}^\ep)_v+\Phi_Y(f_{0,2}^\ep)_v+\Phi_Z(f_{0,3}^\ep)_v
=-v\Phi_X+u\Phi_Y+(u^2+v^2)\Phi_Z.
\end{align*}
Setting
$p=\Phi_X+v\Phi_Z$ and $q=\Phi_Y+u\Phi_Z$
we see $(u,v)$ is an adapted coordinate system,
and $\{p,q\}$ is an associated frame.
Since $g$ is $\RR$-equivalent to $f$,
we have the assertion.
\end{proof}
Since $g$ is $\RR$-equivalent to $f$,
their differential geometric properties
are the same.
Thus a $D_4^-$-singularity satisfies the conculusion of Lemma 
\ref{lem:basispl}.
By \cite[Theorem 1.1]{sajid4},
if $f$ satisfies the conculusion of Lemma 
\ref{lem:basispl}, and $f$ is a front,
then it is a $D_4^-$-singularity.
Moreover, we easily see the germ
$(u,v)\mapsto((u^2-v^2)/2,auv,0)$ satisfies
the conculusion of Lemma 
\ref{lem:basispl}.
We give a normal form for a germ
having an atepted coordinate system, which
is slightly a general object to the $D_4^-$-singularity.

Let us set
\begin{align*}
f_{{\rm n}2}
=&\left(\dfrac{u^2 - v^2}{2},\alpha u v,0\right)\\
f_{{\rm n}3}
=&\dfrac{1}{6}\Big(0,
-b_{12} u^3 + 3 b_{21} u^2 v + 3 b_{12} u v^2 - b_{21} v^3,
-c_{12} u^3 + 3 c_{21} u^2 v + 3 c_{12} u v^2 - c_{21} v^3\Big),\\
f_{{\rm n}4}
=&\dfrac{1}{24}\Big(0,
b_{40} u^4 + 4 b_{31} u^3 v - 4 b_{31} u v^3 + 
    b_{40} v^4,\\
&\hspace{50mm}
c_{40} u^4 + 4 c_{31} u^3 v + 6 c_{22} u^2 v^2 
- 4 c_{31} u v^3 +  c_{40} v^4\Big),\\
f_{{\rm n}5}
=&\dfrac{1}{120}\Big(0,
3 b_{14} u^5  + 5 b_{41} u^4 v 
+ 5 b_{14} u v^4 + 3 b_{41} v^5,\\
&
(3 c_{14} + 2 c_{32}) u^5 + 5 c_{41} u^4 v 
+ 10 c_{32} u^3 v^2+ 
    10 c_{23} u^2 v^3 + 5 c_{14} u v^4 + (2 c_{23} + 3 c_{41}) v^5\Big),
\end{align*}
and let us set
\begin{equation}\label{eq:fi}
f_k=\sum_{j=2}^kf_{{\rm n}j}
\quad(k=2,3,4,5).
\end{equation}
We show the following theorem.
\begin{theorem}\label{thm:normal1}
Let $f:(\R^2,0)\to (\R^3,0)$ be a frontal with $df_0=0$,
whose unit normal vector is $\nu$.
We assume that there exist an adapted coordinate system $(u,v)$.
Then for any $k\in\{3,4,5\}$, 
there exist an orientation-preserving diffeomorphism
$\phi:(\R^2,0)\to(\R^2,0)$ and $T\in SO(3)$ such that
\begin{equation}\label{eq:normald4m}
T\circ f\circ\phi^{-1}(u,v)
=
f_k(u,v)+(a(u,v),b(u,v),c(u,v)),
\end{equation}
where $a,b,c$ are rotaional-radial compatible functions
with respect to $(u,v)$
satisfying $j^ka(0)$ $=$ $j^kb(0)$ $=$ $j^kc(0)=0$, and
$\alpha\geq1$, $c_{21}\geq0$, $c_{12}\geq0$.
\end{theorem}
The right-hand side of \eqref{eq:normald4m} is called
a {\it normal form\/} of a $D_4^-$-singularity.
Let us set
\begin{equation}\label{eq:fiti}
\tilde{f}_k
=
\sum_{j=2}^k\tilde{f}_{{\rm n}j}
\quad(k=3,4,5),
\end{equation}
where each
$\tilde{f}_{{\rm n}j}$ is obtained from
$f_{{\rm n}j}$ by replacing all coefficients 
$\alpha$, $b_{ij}$ $c_{ij}$ 
in its defining expression with
$\tilde{\alpha}$, $\tilde{b}_{ij}$, $\tilde{c}_{ij}$ 
$(i,j\in\{0,1,\ldots,5\})$, respectively.
The uniqueness of the normal form 
holds as the following sense.
\begin{theorem}\label{thm:normal2}
Let $k\in\{3,4,5\}$.
Let $f_k$ be the map defined in \eqref{eq:fi}.
We assume $\alpha>1$ and $c_{21}>0$.
Let $\tilde{f}_k$ be the map defined in \eqref{eq:fiti}.
If there exist an orientation-preserving diffeomorphism
$\phi:(\R^2,0)\to(\R^2,0)$ and $T\in SO(3)$ such that
\begin{equation}\label{eq:normald4m2}
j^k\big(T\circ f_k\circ\phi^{-1}\big)(0,0)
=j^k\big(\tilde{f}_k\big)(0,0),
\end{equation}
satisfying $j^k\tilde{a}(0)=j^k\tilde{b}(0)=j^k\tilde{c}(0)=0$, and
$\tilde{\alpha}>1$, $\tilde{c}_{21}>0$, $\tilde{c}_{12}\geq0$
then $\phi$ and $T$ are identities.
In particular, $\alpha=\tilde{\alpha}$,
${b}_{ij}=\tilde{b}_{ij}$, ${c}_{ij}=\tilde{c}_{ij}$
holds for any 
$(i,j\in\{0,1,\ldots,5\},\ i+j=k)$.
\end{theorem}
We will prove Theorems \ref{thm:normal1} and \ref{thm:normal2}
in the remainder of this section.
The form like the right-hand side of \eqref{eq:normald4m} is called
a {\it normal form\/} of $D_4^-$-singularity.
Such forms for Whitney umbrella is constructed in
\cite{WE}, and 
using the form, geometry of Whitney umbrella
is investigated, see \cite{bw,fh,hhnuy,taripair},
for example.
See \cite{ms,oscross,d4geom,shima1} for other normal forms
and its geometry.
\subsection{Normal form}
\begin{lemma}\label{lem:upvq}
Let $f:(\R^2,0)\to(\R^3,0)$ be a frontal satisfying
$\rank df_0=0$.
Let $(u,v)$ be an adapted coordinate system with
the associated frame $\{p,q\}$.
Then there exists an adapted coordinate system $(x,y)$ 
with the associated frame $\{\tilde p,\tilde q\}$
such that
$|\tilde{p}(0,0)|=1$ and $\tilde{p}(0,0)\cdot \tilde{q}(0,0)=0$.
\end{lemma}
\begin{proof}
We set
\begin{equation}\label{eq:rot}
\pmt{u\\v}=
\pmt{r\cos\theta&-r\sin\theta\\ r\sin\theta&r\cos\theta}
\pmt{x\\y}
\end{equation}
for constants $r$ and $\theta$.
Substituting \eqref{eq:compati1}, \eqref{eq:rot} and their differentials
into 
$f_x=f_uu_x+f_vv_x$ and 
$f_y=f_uu_y+f_vv_y$, we have
\begin{equation}\label{eq:xyuv}
\pmt{f_x\\f_y}
=
r^2\pmt{x&y\\-y&x}
\pmt{\tilde{p}\\ \tilde{q}},
\end{equation}
where
$
\tilde p=r^2(\cos2\theta p+\sin2\theta q)$, and
$\tilde q=r^2(-\sin2\theta p+\cos2\theta q)$.
The inner product satisfies
\begin{align*}
\dfrac{1}{r^4}\tilde p\cdot\tilde q&=
 \dfrac{1}{2}\sin4\theta (-p\cdot p+q\cdot q)
+\cos4\theta p\cdot q.
\end{align*}
If
$p\cdot q\ne0$, then we set $\theta$ satisfying
$$
\dfrac{p\cdot p-q\cdot q}{2p\cdot q}=\dfrac{\cos4\theta}{\sin4\theta}.
$$
Then we see $\tilde p\cdot \tilde q=0$.
Furthermore, 
$
\tilde p\cdot\tilde p=
r^4X
$ holds, where 
$$
X=
\cos^22\theta p\cdot p
    +2\cos2\theta\sin2\theta p\cdot q
    +\sin^22\theta q\cdot q>0.
$$
Thus setting
$
r^4=1/X
$,
we have $\tilde p\cdot\tilde p=1$.
\end{proof}

In this section, we show Theorems \ref{thm:normal1} and \ref{thm:normal2}.
\begin{proposition}\label{prop:normalnotb}
Let $f:(\R^2,0)\to(\R^3,0)$ be a frontal satisfying
$\rank df_0=0$.
We assume that there exist a coordinate system $(u,v)$ and
a frame field $\{p,q\}$ of $\nu^\perp$ 
such that \eqref{eq:compati1} holds.
Then there exist an adapted coordinate system $(x,y)$
and $T\in SO(3)$ such that
$$
T\circ f(x,y)
=
\left(
\dfrac{x^2-y^2}{2},\alpha xy,0\right)
+
(a(x,y),b(x,y),c(x,y)),
$$
where $\alpha\geq1$ and $a,b,c$ 
are rotaional-radial compatible functions
satisfying
$j^2a(0)$ $=j^2b(0)$ $=j^2c(0)=0$.
\end{proposition}
\begin{proof}
By Lemma \ref{lem:upvq}, we may assume 
$p(0,0)=(1,0,0)$ and
$q(0,0)=(0,\alpha,0)$ $(\alpha>0)$.
Then 
$$
j^2f(0)=((x^2-y^2)/2,\alpha x y,0)
$$
holds.
Since $\alpha\geq0$, we set $\alpha=\tilde\alpha^2$.
Rotating $f$ by $\pi/2$ along the axis $(0,0,1)$,
and taking a coordinate change 
$u=(x+y)/(\sqrt{2}\tilde\alpha),
v=(-x+y)/(\sqrt{2}\tilde\alpha)$.
Then $f$ is transformed into 
$((x^2-y^2)/2,x y/\tilde\alpha^2)$. 
Moreover, one can see that
if $(u,v)$ is an adapted coordinate system, then
$(x,y)$ is also an adapted coordinate system.
This means that we may assume $\alpha\geq1$.
Thus $f$ is given by
$$
f(x,y)=\big((x^2-y^2)/2+\tilde{a}(x,y),\alpha x y+\tilde{b}(x,y),\tilde{c}(x,y)\big),\ 
j^2\tilde{a}(0)=j^2\tilde{b}(0)=j^2\tilde{c}(0)=0,
$$
and $(x,y)$ is an adapted coordinate system, where
$\tilde{a},\tilde{b},\tilde{c}$ are 
rotational-radial compatible with respect to $(x,y)$.
This proves the assertion.
\end{proof}
\begin{proof}[Proof of Theorem {\rm \ref{thm:normal1}}]
By Proposition \ref{prop:normalnotb},
we may assume $f(u,v)$ is written by
\begin{equation}\label{eq:normal01}
f(u,v)
=
\left(
\dfrac{u^2-v^2}{2},\alpha uv,0\right)
+
(a(u,v),b(u,v),c(u,v))
\end{equation}
satisfying $j^2a(0,0)=j^2b(0,0)=j^2c(0,0)=0$ and
$a,b,c$ are 
rotational-radial compatible with respect to $(u,v)$.
We write 
$a(u,v)=-a_{12}u^3/3!+a_{21}u^2v/2+a_{12}uv^2/2-a_{21}v^3/3!+O(3)$
(see \eqref{eq:rotradfunc} for the rule of coefficients),
where $O(n)$ stands for the terms whose degrees are greater
than $n$.
We set
$x=u + (a_{12}u^{2} - 2a_{21}uv - a_{12}v^{2})/6$,
$y=v + (a_{21}u^{2} + 2a_{12}uv - a_{21}v^{2})/6$.
Then we see $(x,y)$ is adapted, and
$f(x,y)$ is written by
\begin{equation}\label{eq:normal02}
f(x,y)
=
\left(
\dfrac{x^2-y^2}{2},\alpha xy,0\right)
+
(\tilde{a}(x,y),\tilde{b}(x,y),\tilde{c}(x,y))
\end{equation}
satisfying $j^3\tilde{a}(0,0)=0$, $j^2\tilde{b}(0,0)=j^2\tilde{c}(0,0)=0$ and
$\tilde{a},\tilde{b},\tilde{c}$ are 
rotational-radial compatible with respect to $(x,y)$.
If 
$c_{21}\leq0$ and $c_{12}\leq0$, then taking $(u,v)\mapsto(-u,-v)$,
if 
$c_{21}\leq0$ and $c_{12}\geq0$, then taking $(u,v)\mapsto(-v,u)$
if 
$c_{21}\geq0$ and $c_{12}\leq0$, then taking $(u,v)\mapsto(v,-u)$,
we see $c_{21}\geq0$ and $c_{12}\geq0$ can be satisfied.
Next, we assume $f(u,v)$ is written as in \eqref{eq:normal01}
satisfying $j^3a(0,0)=0$, $j^2b(0,0)=j^2c(0,0)=0$ and
$a,b,c$ are 
rotational-radial compatible with respect to $(u,v)$.
Write 
$a=\sum_{i+j=4}a_{ij}u^iv^j/(i!j!)$ and
$b=\sum_{i+j=3}^4b_{ij}u^iv^j/(i!j!)$
satisfying \eqref{eq:rotradfunc}.
We set
$$
\begin{array}{rl}
x&=u+
\dfrac{1}{24\alpha}\big(-a_{40}\alpha u^{3}-3b_{22}u^{2}v
-3a_{22}\alpha uv^{2}+(4a_{31}\alpha-3b_{22})v^{3}\big),\\[3mm]
y&=v+
\dfrac{1}{24\alpha}\big((4a_{31}\alpha-3b_{22})u^{3}
+3a_{22}\alpha u^{2}v-3b_{22}uv^{2}+a_{40}\alpha v^{3}\big).
\end{array}
$$
Then we see $(x,y)$ is adapted, and
$f(x,y)$ is written by \eqref{eq:normal02}
satisfying $j^4\tilde{a}(0,0)=0$, $j^2\tilde{b}(0,0)=j^2\tilde{c}(0,0)=0$ and
$\tilde{a},\tilde{b},\tilde{c}$ are 
rotational-radial compatible with respect to $(x,y)$.
Moreover, $\tilde{b}_{22}=0$ holds, where $\tilde{b}=\sum_{i+j=2}^4
\tilde{b}_{ij}x^iy^j/(i!j!)$.
Furthr, we assume $f(u,v)$ is written as in \eqref{eq:normal01}
satisfying $j^4a(0,0)=0$, $j^2b(0,0)=j^2c(0,0)=0$ and
$a,b,c$ are 
rotational-radial compatible with respect to $(u,v)$.
Moreover, $b_{22}=0$ holds, where $b=\sum_{i+j=3}^4
b_{ij}u^iu^j/(i!j!)$.
Write $a=\sum_{i+j=5}
a_{ij}u^iu^j/(i!j!)$ satisfying \eqref{eq:rotradfunc}.
We set
\begin{align*}
&C=\\
&\pmt{
0&                                                                 -30\alpha b_{22} \\
-30\alpha b_{22} &                                                 0 \\
0 &                                                                -30\alpha b_{22} \\
-30\alpha b_{22} &                                                 0 \\
 -6a_{14}\alpha^2-4a_{32}\alpha^2 &                                -6a_{23}\alpha^2+ 6a_{41}\alpha^2+15b_{21}b_{22}-10\alpha b_{32} \\
 -6a_{23}\alpha^2-4a_{41}\alpha^2+15b_{21}b_{22}-10\alpha b_{32} & -4a_{14}\alpha^2+14a_{32}\alpha^2+15b_{12}b_{22}-10\alpha b_{23} \\
 -4a_{14}\alpha^2-6a_{32}\alpha^2+15b_{12}b_{22}-10\alpha b_{23} &  6a_{23}\alpha^2+ 4a_{41}\alpha^2+15b_{21}b_{22}-10\alpha b_{32} \\
-14a_{23}\alpha^2+4a_{41}\alpha^2+15b_{21}b_{22}-10\alpha b_{32} &  4a_{14}\alpha^2+ 6a_{32}\alpha^2+15b_{12}b_{22}-10\alpha b_{23} \\
 -6a_{14}\alpha^2+6a_{32}\alpha^2+15b_{12}b_{22}-10\alpha b_{23} &  4a_{23}\alpha^2+ 6a_{41}\alpha^2
}
\end{align*}
and set
$$
(x,y)=(u,v)+\dfrac{1}{240\alpha^2}{}^tC\,
\,\trans{{}\big(
u^3,\ 
u^2v,\ 
uv^2,\ 
v^3,\ 
u^4,\ 
u^3v,\ 
u^2v^2,\ 
uv^3,\ 
v^4
\big)},
$$
where $\trans(~)$ stands for the matrix transposition.
Then we see $(x,y)$ is adapted, and
$f(x,y)$ is written by \eqref{eq:normal02}
satisfying $j^5\tilde{a}(0,0)=0$, $j^2\tilde{b}(0,0)=j^2\tilde{c}(0,0)=0$ and
$\tilde{a},\tilde{b},\tilde{c}$ are 
rotational-radial compatible with respect to $(x,y)$.
Moreover, $\tilde{b}_{22}=\tilde{b}_{32}=\tilde{b}_{23}=0$ holds, where $\tilde{b}=\sum_{i+j=3}^5
\tilde{b}_{ij}x^iy^j/(i!j!)$.
This proves the assertion.
\end{proof}
\begin{proof}[Proof of Theorem {\rm \ref{thm:normal2}}]
Since the subspace $(0,0,1)^\perp$ is independent, 
it holds that
$$T=A_\theta=
\pmt{
\cos\theta&-\sin\theta&0\\
\sin\theta& \cos\theta&0\\
0&0&1},\quad\text{or}\quad
T=A'_\theta=\pmt{
\cos\theta& \sin\theta&0\\
\sin\theta&-\cos\theta&0\\
0&0&-1}.
$$
We set
$\phi(u,v)=(\phi_{110}u+\phi_{101}v, 
\phi_{210}u+\phi_{201}v)$ and $(x,y)=\phi(u,v)$,
where $\phi_{110}$, $\phi_{101}$, $\phi_{210}$, $\phi_{201}\in\R$ and
$\phi_{110}\phi_{201}-\phi_{101}\phi_{210}>0$.
By assumption, $T\circ f_2\circ \phi^{-1}(x,y)=\tilde{f}_2(x,y)$
holds.
Conparing the coefficients of the first and the second components of
$T\circ f_2\circ \phi^{-1}(x,y)$ and $\tilde{f}_2(x,y)$,
we have $T=A_\theta$ and
$(a_{11},a_{12},a_{21},a_{22},\theta)$ $=(1,0,0,1,2\pi)$,
$(a_{11},a_{12},a_{21},a_{22},\theta)$ $=(-1,0,0,-1,2\pi)$,
$(a_{11},a_{12},a_{21},a_{22},\theta)$ $=(0,-1,1,0,\pi)$ or
$(a_{11},a_{12},$ $a_{21},a_{22},\theta)$ $=(0,1,$ $-1,0,\pi)$.
When by the assumption $c_{21}>0$ and $c_{12}\geq0$, 
comparing the third components of $f_{3}(u,v)$ and 
$A\circ f_3\circ s^{-1}(u,v)$,
we see that only the first case is appropriate.
We set
$$
\phi(u,v)=\left(
u+\sum_{i+j=2}^4\dfrac{\phi_{1ij}}{i!j!}u^iv^j,\ 
v+\sum_{i+j=2}^4\dfrac{\phi_{2ij}}{i!j!}u^iv^j\right).
$$
Conparing the coefficients of the first and the second components of
$T\circ f_2\circ \phi^{-1}(x,y)$ and $\tilde{f}_2(x,y)$,
with $T$ is the identity,
we have $\phi_{1ij}=\phi_{2ij}=0$ for any $i,j\in\{0,\ldots,4\}$.
This shows the assertion.
\end{proof}
\begin{remark}
For $D_4^+$-singularity, we set
\begin{align*}
f_{{\rm n}2}=&
\dfrac{1}{2}\left(u^2 - v^2,\alpha(u^2 + v^2),0\right),\\
f_{{\rm n}3}=&
\dfrac{1}{6}\Big(0, b_{30} u^3 + b_{03} v^3,c_{30} u^3 + c_{03} v^3\Big),\\
f_{{\rm n}4}=&
\dfrac{1}{24}\Big(0, b_{40} u^4 + b_{40} v^4,
c_{40} u^4 + 6 c_{22} u^2 v^2 + c_{40} v^4\Big),\\
f_{{\rm n}5}=&
\dfrac{1}{120}\Big(0, (b_{50} u^5   + b_{05} v^5,
c_{50} u^5 + 12 c_{32} u^3v^2 + 12 c_{23} u^2 v^3 
+ c_{05}v^5\Big)
\end{align*}
and
$$
f_k=\sum_{j=2}^if_{{\rm n}j}
\quad(k=2,3,4,5).
$$
If $\alpha>0$, $c_{30}>0$, $c_{03}>0$, then
$f_k$ $(k=3,4,5)$ is another normal form for $D_4^+$-singularity given in
\cite[Theorem 2.1]{sajid4}, with
the uniqueness in the sense of Theorem \ref{thm:normal2}.
This can be shown by the same method as the proof of 
Theorems \ref{thm:normal1} and \ref{thm:normal2}.
\end{remark}

\section{Geometry on $D_4^-$-singularity}
In this section, we study geometry on $D_4^-$-singularity
using the normal form $f_3$ in Theorem \ref{thm:normal1}.

\subsection{Curvatures}
The set $S(f)$ of singular points consists of only
the origin, and the Gaussian curvature $K$ and the
mean curvature can be expanded as
\begin{align}
\label{eq:d4k}
K&=\dfrac{1}{u^2+v^2}
\Big(-\dfrac{c_{30}^2+c_{03}^2}{4\alpha^2}+O(1)\Big),\\
\label{eq:d4h}
H&=\dfrac{1}{u^2+v^2}\Big(\dfrac{\alpha^2-1}{2\alpha}
(c_{30}u+c_{03}v)+O(2)\Big).
\end{align}
It is known that the minimal surface $k$ written by the
Weierstrass representation formula
$$
k={\rm Re}\int\pmt{
(1-g^2)h\\
i(1+g^2)h\\
2gh}dz
$$
have a $D_4^-$-singular point if and only if
$g_zh_z\ne0$.
The formula \eqref{eq:d4h} implies that $\alpha=1$ is a
necessary condition
for the mean curvature does not diverge,
the invariant $\alpha$ is $1$ for $D_4^-$-singularities of
minimal surfaces.
On the other hand, by \eqref{eq:d4k},
the Gaussian curvature $K$ diverges to $-\infty$ at
a $D_4^-$-singularity.
Thus there exists asymptotic curves around
a $D_4^-$-singularity.
Let $f$ be the normal form $f_3$ in Theorem \ref{thm:normal1}.
Let $\nu$ be a unit normal vector field and 
let
$$L=f_{uu}\cdot\nu,\quad
M=f_{uv}\cdot\nu,\quad
N=f_{vv}\cdot\nu$$
be the coefficients of the second fundamental form.
We set a symmetric $(0,2)$ tensor $w_{as}$ by
$$
w_{as}=
L\,du^2+2M\,dudv+N\,dv^2.
$$
A vector field $X\in {\frak X}(\R^2,0)$ is 
{\it a solution of $w_{as}=0$} if
$w_{as}(X,X)=0$ holds and a {\it configration defined by
$w_{as}=0$} is integral curves of solutions of
$w_{as}=0$. We have the following theorem.
\begin{theorem}
Let $f:(\R^2,0)\to (\R^3,0)$ be a frontal with $df_0=0$,
which has an adapted coordinate system.
We assume that $\alpha\ne0$, and $f$ is a front.
Then there exists a diffeomorphism-germ $\Phi$ 
such that the one-jet of $w_{as}$ 
is
$$
j^1(\Phi^* w_{as})
=-v\,du^2-2u\, dudv+v\,dv^2
$$
upto a non-zero multiplication.
\end{theorem}
\begin{proof}
By the assumption, we write $f$ as the form \eqref{eq:normald4m}.
There is no assumption $\alpha\ne1$, the coefficients
$(c_{21},c_{12})$ are not uniquely determined, however, 
the condition
$(c_{21},c_{12})\ne(0,0)$ is determined without the assumption
$\alpha\ne1$.
Moreover, $(c_{21},c_{12})\ne(0,0)$ is equivalent to 
$f$ is a front, we may assume $(c_{21},c_{12})\ne(0,0)$.
By a calculation, $j^1(w_{as})$ is a non-zero multiplication of
$$
(-c_{12}u+c_{21}v)\,du^2+2(c_{21}u+c_{12}v)\,dudv+(c_{12}u-c_{21}v)\,dv^2.
$$
If $c_{12}=0$, then $c_{21}\ne0$ and the assertion holds.
So we assume $c_{12}\ne0$.
If $c_{21}=0$, then the assertion holds by a change $(u,v)\mapsto(-v,u)$.
So we assume $c_{21}c_{12}\ne0$.
We consider a coordinate change
$$
x=-kc_{12}u+ c_{21}v,\quad
y=- c_{21}u-kc_{12}v,
$$
where $k$ is undetermined yet.
Then by a direct calculation, we see
$j^1(w_{as})$ written in the coordinate system $(x,y)$ 
is a non-zero multiplication of
\begin{align*}
&  ( w_{1}x+w_{2}y)\,dx^2
+2( w_{2}x-w_{1}y)\,dxdy
+ (-w_{1}x-w_{2}y)\,dy^2,\\
w_1=&k^3 c_{12}^4+3k^2 c_{12}^2 c_{21}^2-3k c_{12}^2 c_{21}^2-c_{21}^4,\\
w_2=&-c_{12}c_{21}(k^3 c_{12}^2-3k^2 c_{12}^2-3k c_{21}^2+c_{21}^2).
\end{align*}
Since $c_{12}\ne0$, there is a solution $k$ to $w_1=0$, and
since the resultant of $w_1$ and $w_2$ with respect to $k$ is 
$-64\,c_{12}^9 c_{21}^9 (c_{12}^2+c_{21}^2)^3$,
this solution $k$ is not a solution of $w_2=0$.
Choosing $k$ as the above, we have the assertion.
\end{proof}
The configration of
$\omega_0=-v\,du^2-2u\, dudv+v\,dv^2=0$
is given by \cite{BruceFidal} and 
its figure in the source space is drawn in Figure \ref{fig:BruceFidal}
left, and an example of its figure in a $D_4^-$-singularity
is drawn in Figure \ref{fig:BruceFidal}
right.
\begin{figure}[ht]
\begin{center}
\includegraphics[width=.3\linewidth]{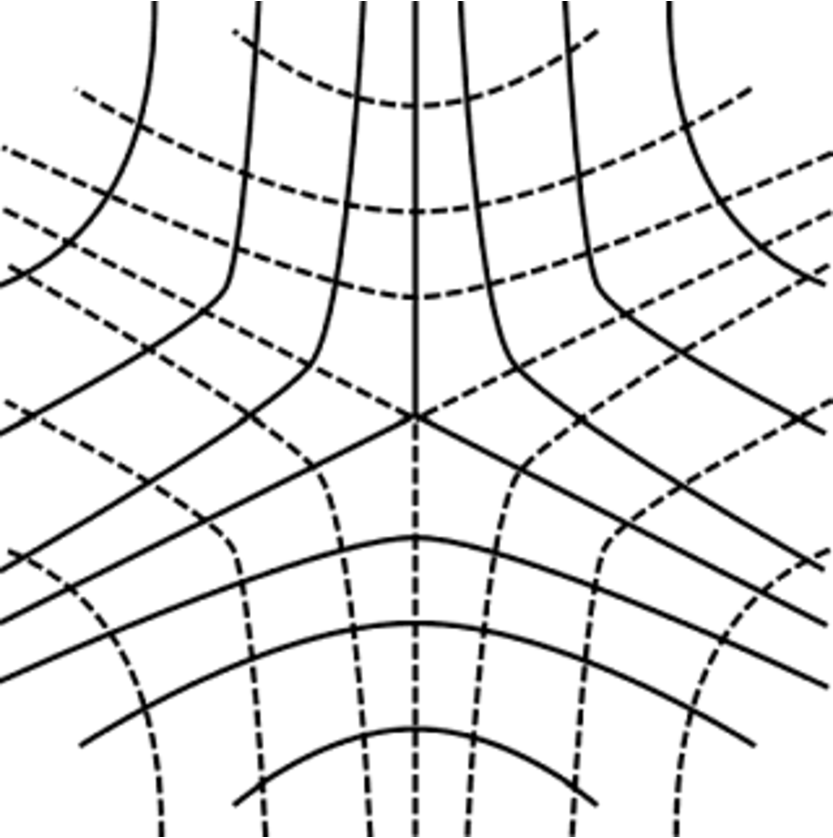}
\hspace{1cm}
\includegraphics[width=.4\linewidth]{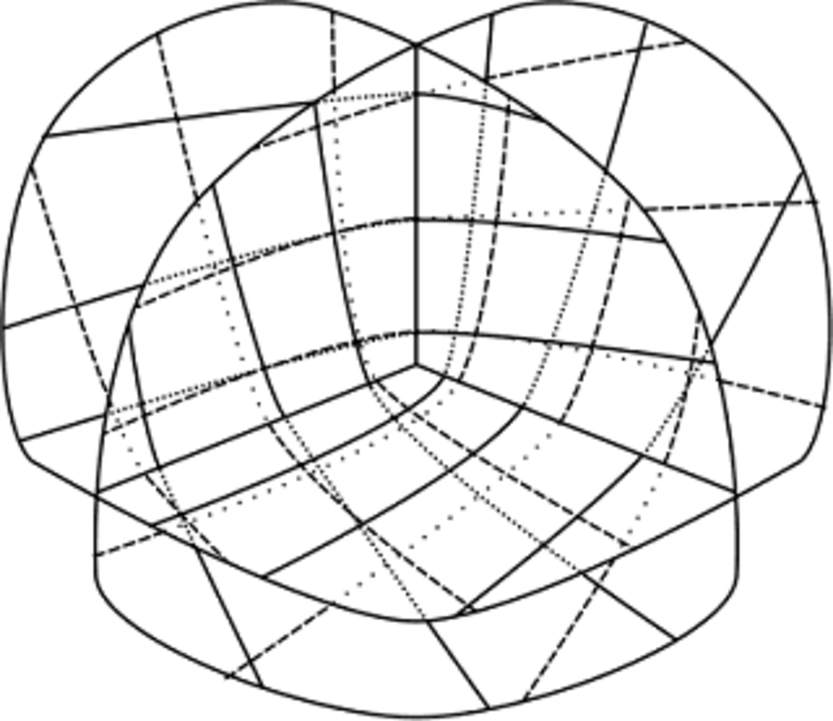}
\end{center}
\caption{Configration of solutions of $\omega_0=0$.}
\label{fig:BruceFidal}
\end{figure}

\subsection{Intrinsity of $\alpha$}
\begin{lemma}\label{lem:upvq2}
Let $f:(\R^2,0)\to(\R^3,0)$ be a frontal satisfying
$\rank df_0=0$.
{\rm (1)} 
Let $(u,v)$ and $(x,y)$ be two adapted coordinate systems
with the associating frames $\{p,q\}$ and $\{\tilde p,\tilde q\}$
respectively.
If 
$|p|=|\tilde{p}|=1$ and
$p\cdot q=\tilde{p}\cdot\tilde{q}=0$ hold at $(0,0)$,
then
it holds that $|q(0,0)|=|\tilde{q}(0,0)|$.
{\rm (2)} 
Let $(u,v)$ and $(x,y)$ be two coordinate systems,
and let $p_0,q_0,\tilde{p}_0,\tilde{q}_0$ be constant vectors
satisfying
\begin{equation}\label{eq:upvqxpxq}
\left\{
\begin{array}{ll}
j^1f_u(0)&= up_0+vq_0,\\
j^1f_v(0)&=-vp_0+uq_0
\end{array}
\right.
\quad
\left\{
\begin{array}{ll}
j^1f_x(0)&= x\tilde{p}_0+y\tilde{q}_0,\\
j^1f_y(0)&=-y\tilde{p}_0+x\tilde{q}_0
\end{array}
\right.
\end{equation}
and 
$|p_0|=|\tilde{p}_0|=1$ and
$p_0\cdot q_0=\tilde{p}_0\cdot\tilde{q}_0=0$ hold.
Then $|q_0|=|\tilde{q}_0|$.
\end{lemma}
\begin{proof}
It is enough to show (2).
By definition, it holds that 
$f_{uu}=-f_{vv}=p_0$,
$f_{uv}=q_0$,
$f_{xx}=-f_{yy}=\tilde{p}_0$,
$f_{xy}=\tilde{q}_0$ at $(0,0)$.
By a rotation, we can assume $p_0=\tilde{p}_0$.
Let $x=x(u,v)$, $y=y(u,v)$ be a coordinate change between $(u,v)$ and $(x,y)$.
Then we obtain
\begin{align*}
f_{uu}&=f_{xx}(x_u)^2+2f_{xy}x_uy_u+f_{yy}(y_u)^2
=f_{xx}((x_u)^2-(y_u)^2)+2f_{xy}x_uy_u\\
f_{uv}&=f_{xx}x_uy_u+f_{xy}(x_uy_v+x_vy_u)+f_{yy}y_uy_v
=f_{xx}(x_uy_u-y_uy_v)+f_{xy}(x_uy_v+x_vy_u),\\
f_{vv}&=f_{xx}(x_v)^2+2f_{xy}x_vy_v+f_{yy}(y_v)^2
=f_{xx}((x_v)^2-(y_v)^2)+2f_{xy}x_vy_v
\end{align*}
at $(0,0)$.
By $p_0=\tilde{p}_0$, we have $x_uy_u=x_vy_v=0$
and $(x_u)^2-(y_u)^2=-((x_v)^2-(y_v)^2)=1$.
If $x_u=0$, then $-(y_u)^2=1$, and this is a contradiction.
So, $y_u=x_v=0$, $x_u=\pm1$ and $y_v=\pm1$.
Thus $f_{uv}=\pm f_{xy}$ at $(0,0)$ and we have the assertion.
\end{proof}
By Lemma \ref{lem:upvq2},
the length $|q|$ does not depend on the choice of the
coordinate system satisfying the assumption of Lemma \ref{lem:upvq2} (2).
Thus it is a geometric invariant,
we set $\alpha=|q|$.
Let $(u,v)$ and $(x,y)$ be two coordinate systems,
and let $p_0,q_0,\tilde{p}_0,\tilde{q}_0$ be constant vectors
satisfying \eqref{eq:upvqxpxq}
and 
$|\tilde{p}_0|=1$ and $\tilde{p}_0\cdot\tilde{q}_0=0$ hold.
We assume $(u,v)$ and $(x,y)$ relate with
\eqref{eq:xyuv}.
By a direct calculation, we see
\allowdisplaybreaks
\begin{align}\label{eq:alphafirst}
\alpha&=(\tilde q\cdot\tilde q)^{1/2}\Big|_{u=v=0}\\
&=
r^2(  \sin^22\theta p\cdot p
    -2\cos2\theta\sin2\theta p\cdot q
    +\cos^22\theta q\cdot q)^{1/2}\Big|_{u=v=0}\nonumber\\
&=
\left(\dfrac{  \sin^22\theta p\cdot p
    -2\cos2\theta\sin2\theta p\cdot q
    +\cos^22\theta q\cdot q}
{\cos^22\theta p\cdot p
    +2\cos2\theta\sin2\theta p\cdot q
    +\sin^22\theta q\cdot q}\right)^{1/2}\Bigg|_{u=v=0}\nonumber\\
&=
\left(
\dfrac{(1-\cos 4\theta)p\cdot p-2\sin4\theta p\cdot q+
(1+\cos 4\theta)q\cdot q}
{(1+\cos 4\theta)p\cdot p+2\sin4\theta p\cdot q+
(1-\cos 4\theta)q\cdot q}\right)^{1/2}\Bigg|_{u=v=0}\nonumber\\
&=
\left(
\dfrac{
p\cdot p (-P + \sgn P \sqrt{P^2 + 1}) - 
 2 p\cdot q + (P + \sgn P \sqrt{P^2 + 1})  q\cdot q
}
{
p\cdot p (P + \sgn P \sqrt{P^2 + 1}) + 
 2 p\cdot q + (-P + \sgn P \sqrt{P^2 + 1})  q\cdot q}
\right)^{1/2}\Bigg|_{u=v=0}\nonumber\\
&=
\Bigg(\dfrac{1}
{
2((p\cdot p)(q\cdot q)- (p\cdot q)^2)
}
\Big[
(p\cdot p)^2 + 2(p\cdot q)^2 + (q\cdot q)^2 \nonumber\\
&\hspace{20mm}
- \sgn(p\cdot p  -  q\cdot q)(p\cdot p  +  q\cdot q)
\sqrt{(p\cdot p  -  q\cdot q)^2+ 4 (p\cdot q)^2} 
\Big]\Bigg)^{1/2}\Bigg|_{u=v=0}\nonumber\\
&=\Bigg(
\dfrac{1}
{2(E_{uu} E_{vv}-E_{uv}^2)}
\Big[
E_{uu}^2+2 E_{uv}^2+E_{vv}^2\nonumber\\
&\hspace{30mm}
-\sgn(E_{uu}-E_{vv})(E_{uu}+E_{vv})
\sqrt{(E_{uu}-E_{vv})^2+4 E_{uv}^2}\Big]
\Bigg)^{1/2}\Bigg|_{u=v=0}\nonumber
\end{align}
where $P=(p\cdot p-q\cdot q)/(2p\cdot q)$, 
$4\theta=\cot^{-1}P$ and $E,F,G$ are the coefficients of
the first fundamental form.
If $p\cdot q=0$, then
$$
(\tilde q\cdot\tilde q)^{1/2}\Big|_{u=v=0}
=
\left(
\dfrac{q\cdot q}
{p\cdot p}\right)^{1/2}\Bigg|_{u=v=0}=
\left(
\dfrac{E_{vv}}
{E_{uu}}\right)^{1/2}\Bigg|_{u=v=0}.
$$
\begin{lemma}\label{lem:intrcoord}
A coordinate system $(u,v)$ satisfies 
\begin{equation}\label{eq:j1upvq}
j^1f_u(0)=up+vq,\quad
j^1f_v(0)=-vp+uq
\end{equation}
if and only if 
that $j^2\sqrt{EG-F^2}(0,0)$ is a non-zero multiple of
$u^2+v^2$, where $E,F,G$ are the coefficients of the
first fundamental form.
In particular, 
whether a coordinate system $(u,v)$ satisfies
\eqref{eq:j1upvq}
is intrinsically decidable.
\end{lemma}
\begin{proof}
Since $EG-F^2$ is a square of the signed area density,
the property
\eqref{eq:j1upvq} is equivalent to
that $j^2\sqrt{EG-F^2}(0,0)$ is a non-zero multiple of
$u^2+v^2$.
\end{proof}
We have the following theorem.
\begin{theorem}
Let $f:(\R^2,0)\to(\R^3,0)$ be a frontal satisfying
$\rank df_0=0$.
We assume that there exists an adapted coordinate system.
Then the invariant $\alpha$ is intrinsic.
\end{theorem}
\begin{proof}
By Lemma \ref{lem:intrcoord} and
the formula \eqref{eq:alphafirst},
we obtain the assertion.
\end{proof}

\subsection{Intersection curves}
In this section,
by studying the behavior of the intersection curves near 
the singular point, we clarify another geometric meaning of $\alpha$.
We assume $f$ is a $D_4^-$-singularity and written as 
the form \eqref{eq:normald4m} with $k=3$.
Let us set $(u,v)=(r\cos\theta,r\sin\theta)$ and
$i(t)=(r(t)\cos\theta(t),r(t)\sin\theta(t))$.
Let
$\hat i=f\circ i$
be a self-intersection curve,
we assume $\hat i(t)=\hat i(-t)$.
By a change of paramter, we may assume $r'(0)=1$,
since $f$ is $\A$-equivalent to $f_0$,
and the self-intersection curves $\hat i_0$ of $f_0$ satisfy
$\hat i_0(t)=\hat i_0(-t)$.
We write $r(t)=t+r^2t^2+O(3)$ and
$\theta(t)=\theta_0+\theta_1t+O(2)$.
Then by $\hat i(t)=\hat i(-t)$,
we have
\begin{align*}
\theta_0=&\dfrac{1}{3}(\tan^{-1}(c_{21},c_{12})+\pi j)\quad
(j=0,1,2,\ldots,5),\\
\theta_1=&
\dfrac{b_{21} c_{12}-b_{12} c_{21}}{2 \alpha \sqrt{c_{12}^2+c_{21}^2}}
\cos\Big(\dfrac{2}{3} (\tan^{-1}(c_{21},c_{12})+\pi j)\Big),\\
r_2=&
\dfrac{-(b_{21} c_{12}-b_{12} c_{21}) }{2 \alpha \sqrt{c_{12}^2+c_{21}^2}}
\sin\Big(\dfrac{2}{3} (\tan^{-1}(c_{21},c_{12})+\pi j)\Big),
\end{align*}
where  $\tan^{-1}(c_{21}, c_{12})$ is defined as 
the two-argument arctangent, giving the oriented angle 
from $(1,0)$ to $(c_{21}, c_{12})$.
Thus the ray tangent to $\hat i(t)$ at $t=0$ is
in the positive directions of 
$$
V_j=\left(
\cos\left(\dfrac{2}{3}(\tan^{-1}(c_{21},c_{12})+2\pi j)\right),
\alpha \sin\left(\dfrac{2}{3} (\tan^{-1}(c_{21},c_{12})+2\pi j)\right),
0\right),
$$
where $j=0,1,2$.
The area of three points $V_j/|V_j|$ $(j=0,1,2)$ is
the same as the area of the triangle generated by
$\{(1,0)$, $(1,\sqrt{3}\alpha)/(1+3\alpha^2)$, $(-1,\sqrt{3}\alpha)/(1+3\alpha^2)\}$,
and the absolute value of it is 
\begin{equation}\label{eq:triarea}
\frac{\sqrt3\alpha\bigl(1+\sqrt{1+3\alpha^2}\bigr)}{1+3\alpha^2}.
\end{equation}
\subsection{Curvatures along curves}
We consider the geodesic curvature of 
a loop encircling the singular point
(respectively, a curve passing through the singular point)
$c:\theta\mapsto(r\cos\theta,r\sin\theta)$ and
$\hat c=f(c(\theta))$
(respectively,
$g:r\mapsto(r\cos\theta,r\sin\theta)$ and
$\hat g=f(g(r))$).
See \cite{hara,haya} for studies of this kind on other singularities.
At a singular point, the geodesic curvature usually unbounded,
we calculate the geodesic curvature measure instead
as in Appendix \ref{sec:kgkn}, they are bounded.
We set
$\tilde\kappa_g\,d\theta=\kappa_g\,ds$
(respectively, 
$\hat\kappa_g\,dr=\kappa_g\,ds$)
for $\hat c(\theta)$
(respectively, 
$\hat g(r)$),
where $s$ is an arclength of 
$\hat c(\theta)$
(respectively, $\hat g(r)$),
and we set
$\tilde\kappa_n\,d\theta=\kappa_n\,ds$
(respectively, 
$\hat\kappa_n\,dr=\kappa_n\,ds$)
for $\hat c(\theta)$
(respectively, 
$\hat g(r)$),
where $s$ is an arclength of 
$\hat c(\theta)$
(respectively, $\hat g(r)$).
The coordinate system $(u,v)$ in 
Theorem \ref{thm:normal1} is unique if $\alpha>1$,
the parameter $(r,\theta)$ has a meaning.
By a direct calculation, we see
$\hat c'|_{r=0}=0$
and
$\hat c''|_{r=0}=(\cos2\theta,\alpha \sin2\theta)/2$,
where $'=d/d\theta$
(respectively,
$\hat g'|_{r=0}=0$
and
$\hat g''|_{r=0}=(-\sin2\theta,\alpha \cos2\theta,0)$,
where $'=d/d\theta$),
and $\hat c$ and $\hat g$ are
frontals as curves on a frontal.
Thus we have
\begin{equation}\label{eq:kg2dir}
\tilde\kappa_g|_{r=0}=
\dfrac{4 \alpha}{3 (1+\alpha^2+(-1+\alpha^2) \cos 4\theta)},
\quad
\hat\kappa_g|_{r=0}=
\dfrac{\cos2 \theta (b_{12} \cos 3 \theta-b_{21} \sin 3\theta)}
{-1-\alpha^2+(-1+\alpha^2) \cos 4\theta}
\end{equation}
and
\begin{equation}\label{eq:kn2dir}
\tilde\kappa_n|_{r=0}=0,
\quad
\hat\kappa_n|_{r=0}=
\dfrac{-3 (c_{12}\cos 3\theta - c_{21}\sin 3\theta)}
{\sqrt{2}\,\sqrt{1+\alpha^{2} - (-1+\alpha^{2})\cos 4\theta}}.
\end{equation}

We set $k_g(\theta)$ (respectively, $\kappa_n(\theta)$)
be the right hand side of 
$\hat\kappa_g|_{r=0}$ in \eqref{eq:kg2dir},
(respectively, $\hat\kappa_n|_{r=0}$ in \eqref{eq:kn2dir})
and calculate the value of $k_g$ and $k_n$ at the
initial angles of intersection curve and its
the angular midpoints:
$$
\psi_{1j}=\dfrac{1}{3}(\tan^{-1}(c_{21},c_{12})+\pi j),\ 
\psi_{2j}=\dfrac{1}{3}(\tan^{-1}(c_{21},c_{12})+\pi/2+\pi j)\ 
(j=0,1,2,\ldots,5).
$$
We set $z=\sqrt{c_{12}^2 + c_{21}^2}$
and
$\phi = \arctan(c_{21},c_{12})$.
Then it holds that
$\cos\phi = c_{12}/z$,
$\sin\phi = c_{21}/z$.
Thus
\[
\cos 3\psi_{1j} = (-1)^j\frac{c_{12}}{z},\ 
\sin 3\psi_{1j} = (-1)^j\frac{c_{21}}{z},\ 
\cos 3\psi_{2j} = (-1)^{j+1}\frac{c_{21}}{z},\ 
\sin 3\psi_{2j} = (-1)^j\frac{c_{12}}{z}.
\]
Hence we have
\begin{equation}\label{eq:kgknint}
\begin{array}{rcl}
\displaystyle
k_g(\psi_{1j})
&=&
\displaystyle
\frac{
(-1)^j\,(b_{12}c_{12} - b_{21}c_{21})\,
\cos\!\left(2\psi_{1j}\right)
}
{
z\Big(-1-\alpha^2
+
(-1+\alpha^2)\cos\!\left(4\psi_{1j}\right)\Big)
},\\
k_g(\psi_{2j})
&=&
\displaystyle
\frac{
(-1)^j\,(-b_{12}c_{21} - b_{21}c_{12})\,
\cos\!\left(2\psi_{2j}\right)
}
{
z\Big(-1-\alpha^2
+
(-1+\alpha^2)\cos\!\left(4\psi_{2j}\right)\Big)
},\\
k_n(\psi_{1j})
&=&
\displaystyle
\dfrac{-3(-1)^{j}(c_{12}^{2}-c_{21}^{2})}
{\sqrt{2}z\sqrt{1+\alpha^{2}-(-1+\alpha^{2})\cos\!\left(4\psi_{1j}\right)}},\\
k_n(\psi_{2j})
&=&
\displaystyle
\dfrac{6(-1)^{j}c_{12}c_{21}}
{\sqrt{2}z\sqrt{1+\alpha^{2}-(-1+\alpha^{2})\cos\!\left(4\psi_{2j}\right)}}.
\end{array}
\end{equation}
The formulas \eqref{eq:alphafirst} or \eqref{eq:triarea}
with
\eqref{eq:kgknint} express the geometric meaning of 
the coefficients $\alpha,b_{21},b_{12},c_{21},c_{12}$ of the
form in Theorem \ref{thm:normal1}.

\section{Gauss-Bonnet type theorem}
A map $f:M\to \R^3$ 
between $2$-dimensional closed manifold $M$ and $\R^3$
is called 
a {\it frontal\/} (respectively, a {\it front})
if for any $p\in M$, the germ $f$ at $p$ is a 
a frontal (respectively, front), namely, 
a unit normal vector $\nu$ is defined locally.
A frontal (respectively, front) is coorientable if
the domain of a unit normal vector $\nu$ can be extended to whole $M$.
Gauss-Bonnet type theorems for coorientable fronts are obtained in
\cite{suyfront, suykyushu, suycoh}, and it is generalized 
to the case of $\partial M\ne\emptyset$ in \cite{dz}.
In these theorems, it is assumed that all singularities $p$ of $f$ 
satisfy $\rank df_p=1$.
Here, we show a Gauss-Bonnet type theorem for coorientable
fronts with $D_4^\pm$-singularities.
A front $f:M\to\R^3$ is called a {\it $1$-parameter-generic front},
if for any singular point $p\in S(f)$ is 
cuspidal edge, swallowtail, cuspidal lips/beaks, cuspidal butterfly
or $D_4^\pm$-singularities.
We follow the proof of the theorem which
is given in \cite[Section 2]{suyfront} and in \cite[Sections 2,3]{suykyushu}.
Since integral terms appear in the Gauss-Bonnet type theorem,
we see the boundedness of the geodesic curvature, 
singular curvature measure and
Gaussian curvature measure for fronts.
\subsection{Boundedness of several measures}

We set the area form $dA$ by
$dA=\sqrt{EG-F^2}\,du\wedge dv=|\lambda|\,du\wedge dv$
for a coordinate system\/ $(u,v)$,
where $\lambda$ is the signed area density function
$\lambda=\det(f_u,f_v,\nu)$ defined in 
\eqref{eq:lambda}.
We also set $d\hat{A}$ by
$d\hat{A}=\lambda\,du\wedge dv$.
\begin{lemma}\label{lem:kgbdd}
Let\/ $f:(\R^2,0)\to(\R^3,0)$ be a frontal, and\/ $\rank df_0=0$.
{\rm (1)}
A curve\/ $\gamma:(\R,0)\to(\R^2,0)$ be a curve 
of finite multiplicity.
Let\/ $\kappa_g$ be the geodesic curvature of\/ $\gamma$.
Then\/
$
\kappa_g\,ds
$
is a bounded\/ $1$-form, 
where\/ $s$ is an arc-length parameter of\/ $f\circ \gamma$, namely,
$ds=|f\circ \gamma'|\,dt$ for a parameter\/ $t$.
{\rm (2)}
Let\/ $f:(\R^2,0)\to(\R^3,0)$ be a frontal, and\/ $\rank df_0=0$.
The Gaussian curvature form\/ $K\,dA$ and\/ $K\,d\hat{A}$ 
for a coordinate system\/ $(u,v)$
are bounded\/ $2$-forms.
\end{lemma}
\begin{proof}
(1) Since $\gamma$ is finite multiplicity, one can set
$(f\circ \gamma)'=t^n x(t)$ by using a non-zero vector field $x$.
Then by a direct calculation, we see the boundedness
of $\kappa_g\,ds$.
(2) Setting 
$$X=\pmt{f_u\\ f_v\\ \nu},$$
it is easy to see the Gaussian curvature $K$ 
with the area form $dA$ satisfies
$$
K\,dA=
\dfrac{\det\left(
X(\nu_u, \nu_v,\nu)\right)}
{(\det X)^2}\,\det X\,du\wedge dv.
$$
It shows the assertion. See \cite[Lemmas 4.1, 4.2]{d4geom} for
datail.
\end{proof}
Since the singular curvature is defined as the geodesic curvature
with a certain sign ({}\cite[Section 2]{suyfront}), 
the boundedness of singular curvature follows from this Lemma.

\subsection{Initial vectors and inner angles}
We show the continuity of initial vector.
Let 
$f:(\R^2,0)\to(\R^3,0)$ be a $D_4^-$-singularity, and
let $\gamma:(\R,0)\to(\R^2,0)$ be a regular curve.
Let 
$c_t:(\R,0)\to(\R^2,\gamma(t))$ be a regular curve
starting from $\gamma(t)$.
We set $\hat\gamma=f\circ \gamma$, $\hat c_t=f\circ c_t$, and
\begin{equation}\label{eq:initial}
\Psi_c(\gamma(t))
=
\dfrac{\dfrac{d\hat{c}_t(w)}{dw}\Big\vert_{w=0}}
{\left|\dfrac{d\hat{c}_t(w)}{dw}\Big\vert_{w=0}\right|}.
\end{equation}
The vector $\Psi_c(\gamma(t))$ is called the {\it initial vector} of
$c_t$ at $\gamma(t)$. 
We remark since $c$ can be taken as $c_t(w)=\gamma(w+t)$,
this notion includes
the tangent vector of $\gamma$.
See Figure \ref{fig:initial} for these settings.
\begin{figure}[ht]
\begin{center}
\includegraphics[width=.3\linewidth]{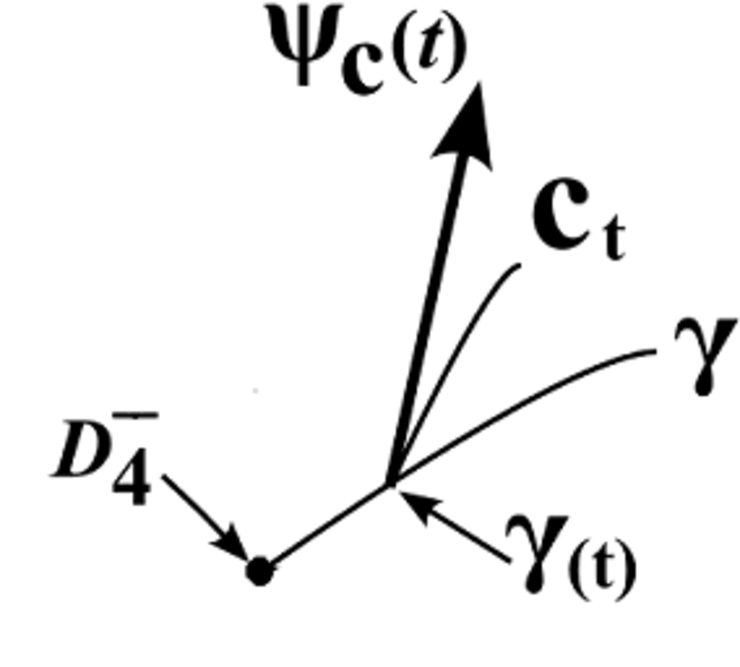}
\end{center}
\caption{Initial vectors}
\label{fig:initial}
\end{figure}

\begin{lemma}[Continuity of the initial vector]\label{lem:contvec}
Under the above assumption,
the initial vector\/ $\Psi_c(t)$ is continuous 
for\/ $t\geq0$ and\/ $\Psi_c(0)\ne0$.
\end{lemma}
\begin{proof}
One can assume $f$ is given by the right-hand side of \eqref{eq:normald4m}
with $k=3$.
Let us set 
$\gamma=(\gamma_1,\gamma_2)$, and $c_t=(c_1(t,w),c_2(t,w))$
satisfying $c_t(0)=\gamma(t)$.
Then 
$$
\dfrac{d\hat{c}_t(w)}{dw}\Big\vert_{w=0}
=
(p,q)\pmt{
(c_1)_w&-(c_2)_w\\
(c_2)_w& (c_1)_w}
\pmt{c_1\\ c_2}(t,0)
$$
holds.
Since $c$ is a regular curve, there exists
$\tilde c(t)=(\tilde c_1(t),\tilde c_2(t))$ such that
$c(t)=t\tilde c(t)$, where
$\tilde c(0)\ne(0,0)$.
Then 
$$\dfrac{d\hat{c}_t(w)}{dw}\Big\vert_{w=0}
\text{ is parallel to }
(p,q)\pmt{
(c_1)_w&-(c_2)_w\\
(c_2)_w& (c_1)_w}
\pmt{\tilde c_1\\ \tilde c_2}(t,0)
$$
Since $c$ is regular, 
$(c_1)_w^2+(c_2)_w^2\ne0$ holds,
and this implies 
the above vector does not vanish.
This proves the assertion.
\end{proof}

Let $f:(\R^2,0)\to (\R^3,0)$ be a frontal with $df_0=0$,
whose unit normal vector is $\nu$.
We assume that there exist an adapted coordinate system $(u,v)$,
and $f$ is written by \eqref{eq:normald4m} with $k=3$.
We calculate the initial vector concretely.
Let $\gamma:([0,\ep),0)\to(\R^2,0)$ be a curve
defined by
$\gamma(t)=(r(t)\cos\theta(t),r(t)\sin\theta(t))$
for $\ep>0$
emanating from the origin.
We set $\hat\gamma=f\circ \gamma$.
Then we see
\begin{align*}
\hat\gamma&=\dfrac{r(t)^2}{2}\Big((\cos2\theta,\alpha\sin2\theta,0)+O(r)\Big),\\
\hat\gamma'&=r(t)\Big((r'\cos2\theta,\alpha r'\sin2\theta,0)+O(r)\Big).
\end{align*}
Thus
\begin{equation}\label{eq:psigamma}
\Psi_\gamma(0)=\dfrac{(\cos2\theta,\alpha\sin2\theta,0)}
{\cos^22\theta+\alpha^2\sin^22\theta}.
\end{equation}

\subsection{Summension of inner angles}
Let $f:M\to \R^3$ be a one-parameter generic front.
A singular point $p\in S(f)$ is said to be of {\it corank one\/}
(respectively, {\it corank two\/})
if $\rank df_p=1$ (respectively, $\rank df_p=0$) holds.
A curve-germ $\gamma:([0,\ep),0)\to (M,p)$ is said to be
{\it admissible\/} if $\gamma$ is regular, and
one of the following holds:
\begin{enumerate}
\item $\gamma([0,\ep))\cap S(f)=\emptyset$;
\item $p$ is a corank one singular point, 
      $\gamma((0,\varepsilon))\subset M\setminus S(f)$, and 
      $\gamma'(0)$ is neither tangent to $S(f)$ nor a null vector at $p$;
\item $p$ is a corank two singular point, 
      $\gamma((0,\varepsilon))\subset M\setminus S(f)$, and 
      $\gamma'(0)$ is not an isotropic direction of $\lambda$ at $p$, 
      that is, not a vector along which the quadratic part of $\lambda$ 
      at $p$ vanishes, where $\lambda$ is a signed area density function;
\item $\gamma([0,\varepsilon))\subset S(f)$.
\end{enumerate}
Let $\gamma:([0,\ep),0)\to (M,p)$ be an admissible curve-germ,
and let $c:([0,\ep),0)\to (M,\gamma(t))$ be an admissible curve-germ.
Then the continuity of
initial vector
$\Psi_c(\gamma(t))$ holds.
See \cite[Proposition 2.6]{suykyushu} for $p$ is a corank one singular point,
and see \cite[Lemma 4.3]{d4geom} for $p$ is a $D_4^+$-singular point.
A curve $\gamma:[a,b]\to M$ is said to be {\it admissible\/} if
the curve-germ $\gamma$ at $a$ and at $b$ is admissible,
and if $\gamma((a,b))\cap S(f)\ne\emptyset$, then $\gamma((a,b))\subset S(f)$.
Let $\gamma_i:([0,\ep),0)\to (M,p)$ $(i=1,2)$ be two admissible curve-germs.
We assume the angle of $\gamma_1'(0)$ and $\gamma_2'(0)$ is
less than $\pi/2$.
Then the angle
$$
\cos^{-1}\left(
\dfrac{ \Psi_{\gamma_1}(\gamma_1(t))\cdot\Psi_{\gamma_2}(\gamma_2(t))}
{|\Psi_{\gamma_1}(\gamma_1(t))|\,| \Psi_{\gamma_2}(\gamma_2(t))|}\right)
$$
is denoted by $\angle_{\gamma_1,\gamma_2}p$ or $\angle p$,
and called the {\it inner angle of\/ $\gamma_1$ and\/ $\gamma_2$}.
For the well-definedness of the inner angle, we show the following lemma.
\begin{lemma}\label{lem:angle}
Let\/ $\gamma_i:([0,\ep),0)\to(\R^2,0)$ $(i=1,2)$ be two curves
defined by\/
$\gamma_i(t)=(r_i(t)\cos\theta_i(t),r_i(t)\sin\theta_i(t))$
emanating from the origin.
We set\/ $\hat\gamma_i=f\circ \gamma_i$.
If\/ $\theta_2-\theta_1<\pi/(2\alpha)$, then
the angle between
$$
\Psi_{\gamma_1}(0)\quad\text{and}\quad
\Psi_{\gamma_2}(0)
$$
is less than\/ $\pi$.
\end{lemma}
\begin{proof}
In the formula \eqref{eq:psigamma},
we consider $\theta$ as a variable, and set
$V(\theta)=(\cos2\theta,\alpha\sin2\theta)$.
We set the angle between $V(\theta)$ and $(1,0)$ as $W(\theta)$.
Then we have
$$
W(\theta)=\arctan \Big(\alpha\tan 2\theta\Big)+n\pi$$
where $n$ is an integer.
We have
$$
W'(\theta)=\dfrac{d}{d\theta}W(\theta)
=
\dfrac{2\alpha(1+t)}{1+\alpha^2 t}, \big(t=\tan^22\alpha\geq0\big).
$$
By a calculus, since
$dW'(\theta)/dt=-2 \alpha (\alpha^2-1)/(1 + \alpha^2 t)^2\leq0$ and
$$
\lim_{t\to0}W'(\theta)=2\alpha,\quad
\lim_{t\to\infty}W'(\theta)=\dfrac{2}{\alpha},
$$
we obtain $2/\alpha\leq W'(\theta)\leq 2\alpha$.
By the mean value theorem, for any $\theta_0$ and $\theta_1$ satisfying
$\theta_0<\theta_1$,
there exists $\xi$ such that
$$
\dfrac{W(\theta_1)-W(\theta_0)}{\theta_1-\theta_0}=W'(\xi).
$$
This implies that 
$$
|W(\theta_1)-W(\theta_0)|< 2\alpha \dfrac{\pi}{2\alpha}=\pi
$$
under the assumption $\theta_1-\theta_0<\pi/(2\alpha)$.
By \eqref{eq:psigamma}, since the angle of $\Psi_{\gamma_i}(0)$ is
$W(\theta_i)$, 
we have the assertion.
\end{proof}

\begin{lemma}
Let\/ $f:(\R^2,0)\to(\R^3,0)$ be a frontal, and\/ $\rank df_0=0$.
We assume that there exists an adapted coordinate system.
Let $\gamma_i:([0,\ep),0)\to(\R^2,0)$ $(i=1,2)$ be distinct curves
defined by
$\gamma_i(t)=(r_i(t)\cos\theta_i(t),r_i(t)\sin\theta_i(t))$
emanating from the origin, satisfying $\theta_1(0)<\ldots<\theta_n(0)$.
If $\theta_{i+1}-\theta_i<\pi/(2\alpha)$ for any $i\in\{1,\ldots,n-1\}$,
then the total inner angle satisfy
$$
\sum_{i=1}^{n-1}
\Big(\Psi_{\gamma_{i+1}}(0)-\Psi_{\gamma_i}(0)\Big)
=4\pi.
$$
\end{lemma}
\begin{proof}
We may assume $f$ is written as the right-hand side of
\eqref{eq:normald4m}.
Then all the initial vectors $\Psi_{\gamma_{i+1}}(0)$ $(i\in\{1,\ldots,n-1\})$
are in the plane $(0,0,1)^\perp$.
Since 
$\theta_{i+1}-\theta_i<\pi/(2\alpha)$ for any $i\in\{1,\ldots,n-1\}$,
and Lemma \ref{lem:angle},
the total sum of angles are equal to the winding number 
of the curve $((u^2-v^2)/2,\alpha u v)_{(u,v)=(r\cos\theta,r\sin\theta)}$ 
$(\theta\in[0,2\pi])$
for sufficiently small $r$ with respect to $(0,0)$.
Since it does not depend on the scaling $(x,y)\mapsto(x,y/\alpha)$,
the total sum of angles are $4\pi$.
\end{proof}

\subsection{Gauss-Bonnet type theorem}
Let\/ $M$ be a closed oriented surface, and
let\/ $f:M\to\R^3$ be a coorientable $1$-parameter-generic front with
a unit normal vector\/ $\nu$ defined on\/ $M$.
For an oriented local coordinate system $(u,v)$,
let $\lambda=\det(f_u,f_v,\nu)$ be a signed area density function.
We define $M_\pm=\{p\in M\,|\, \pm\lambda>0\}$.
It is known that for a singular point $p$ of
$1$-parameter-generic front satisfying $\rank df_p=1$ 
except for cuspidal edges
is
a {\it peak}. For a peak $p$, the {\it sign}, positive, zero or negative
is defined and it is denoted by $\sigma(p)\in\{-1,0,1\}$. 
See \cite[Section 2]{suykyushu} or \cite[Section 2]{suyfront}
for datail.
We set $P_\pm=\{p\in M\,|\,p\text{ is a peak, and }\sigma(p)=\pm1\}$,
and
$D_{4\pm}^-=\{p\in M\,|\,p\text{ is a }D_4^-\text{-singularity, and 
a punctured neighbourhood of $p$ lies in}
M_{\pm}\}$.
We also set $D_{4}^-=D_{4+}^-\cup D_{4-}^-$.

The initial vector of singular curves
emanating from a $D_4^+$-singularity $q$ is well-defined and
angle of two of such curves are the angle of the $D_4^+$-singularity,
which is denoted by $\theta(q)$.
See \cite[(2.32)]{d4geom} for datail.
We have the following theorem:
\begin{theorem}\label{thm:gb}
Let\/ $M$ be a closed oriented surface, and
let\/ $f:M\to\R^3$ be a coorientable $1$-parameter-generic front with
a unit normal vector\/ $\nu$ defined on\/ $M$.
Let\/
$\{d_1,\ldots,d_k\}$ be the set of\/ $D_4^+$-singularities,
and $\theta(d_i)$ be the angle of $D_4^+$-singularity.
Then it holds that
\begin{align*}
\int_MK\,dA+2\int_{S(f)}\kappa_s\,ds
=&
2\pi\chi(M)
+\sum_{i=1}^k(4\theta(q_i)-2\pi)
+2\pi\sharp(D_4^-),
\\
\int_MK\,d\hat{A}
=&
2\pi\Big(\chi(M_+)-\chi(M_-)\Big)
+
2\pi\Big(\sharp(D_{4+}^-)-\sharp(D_{4-}^-)\Big)\\
&\hspace{50mm}
+
2\pi\Big(\sharp(P_+)-\sharp(P_-)\Big)
\end{align*}
where\/ $\sharp(S)$ is the cardinality of the set $S$.
\end{theorem}
This is shown by the following local version of this theorem.
Let\/ $M$ be a closed oriented surface, and
let\/ $f:M\to\R^3$ be a coorientable $1$-parameter-generic front.
\begin{definition}
Let $\overline{T}$ be the closure of a simply connected domain $T$
which is bounded by three admissible curves 
$\gamma_i:[a,b]\to\partial T$ $(i=1,2,3)$.
Let 
$A=\gamma_2(b)=\gamma_3(a)$,
$B=\gamma_3(b)=\gamma_1(a)$,
$C=\gamma_1(b)=\gamma_2(a)$ be three 
intersections of $\gamma_i$ $(i=1,2,3)$.
We denote by this situation 
$T=\triangle ABC$, and say triangle $T$.
A triangle
$T=\triangle ABC$ is an
{\it admissible triangle} if 
$\gamma_i$ $(i=1,2,3)$ are admissible
and $T^\circ\cap S(f)=\emptyset$, and
the interior angles $\angle A$, $\angle B$, $\angle C$ as angles in $M$ 
is less than $\pi/2$.
Here $T^\circ$ stands for the interior of $T$.
\end{definition}
Then the following holds:
\begin{lemma}\label{lem:locgb}
For an admissible triangle $T=\triangle ABC$,
the following holds:
\begin{equation}\label{eq:locgb}
\angle A+\angle B+\angle C-\pi
=
\int_{\partial T}\hat\kappa_g\,ds
+
\int_{T}K\,dA.
\end{equation}
where $s$ is an arc-length parameter
of the corresponding curve.
Here,
$$
\hat\kappa_g=\left\{
\begin{array}{ll}
\kappa_g &\quad(\text{on}\ M_+),\\
-\kappa_g&\quad(\text{on}\ M_-),\\
\kappa_s &\quad(\text{on}\ S(f)).
\end{array}
\right.
$$
\end{lemma}
\begin{proof}
By considering subdivision of $T$,
it suffices to consider only the cases 
where at most one edge is a cuspidal edge and 
at most one vertex is a singular point other than a cuspidal edge.
See \cite[Section 2]{suykyushu} for 
the case $T\cap S(f)$ consists only of singularities 
of corank one.
See \cite[Proof of Theorem 4.4]{d4geom} 
for the case one point of $A,B,C$ is a
$D_4^+$-singular point and
at most one edge is a cuspidal edge.
Thus it is enough to show that the case that 
$A$ is a $D_4^-$-singular point,
and $(T\setminus\{A\})\cap S(f)=\emptyset$.
We take $A_t$ $(t\in[0,\ep))$ on the edge $AB$ where $A_0=A$,
see Figure \ref{fig:triangles}.
\begin{figure}[ht]
\begin{center}
\includegraphics[width=.2\linewidth]{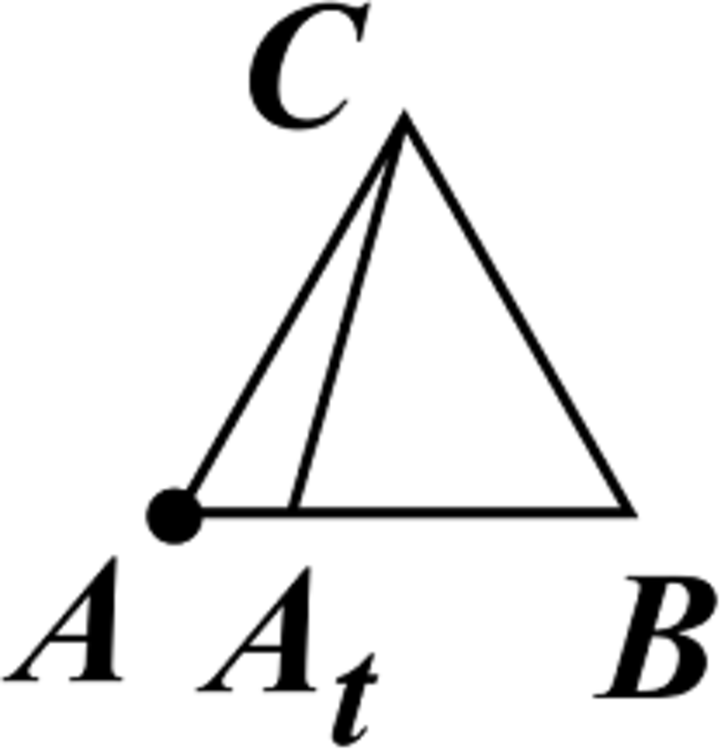}
\end{center}
\caption{Triangle}
\label{fig:triangles}
\end{figure}
Then \eqref{eq:locgb} holds for $\triangle A_tBC$ 
(\cite[Section 2]{suykyushu}).
By Lemma \ref{lem:contvec},
the initial vector of the edges $A_tB$, $A_tC$ at $A_t$ are well-defined
and converges to that of the edges $AB$, $AC$ at $A$ respectively.
Thus \eqref{eq:locgb} holds for $\triangle ABC$.
\end{proof}
\begin{proof}[Proof of Theorem\/ {\rm \ref{thm:gb}}]
Taking a triangulation of $M$ satisfying
all triangles are admissible, and
at most one edge is a cuspidal edge and 
at most one vertex is a singular point other than a cuspidal edge.
By Lemma \ref{lem:locgb},
the formula \eqref{eq:locgb}
holds for all triangles.
We look at a $D_4^-$-singularity $p$.
We assume $p$ is a vertex of a triangle $T$.

In the usual proof of the Gauss-Bonnet theorem, 
a regular point $q$ contributes a total interior 
angle of $2\pi$ from all triangles having $q$
as a vertex.  
In contrast, a $D_4^-$-singularity $p$ contributes $4\pi$, so that each
$p\in D_{4\pm}^-$ gives an excess contribution of $2\pi$.
Consequently, when taking the sum (respectively, the difference) of
\eqref{eq:locgb} over all triangles 
contained in $M_+$ and $M_-$, we obtain
the additional term
$2\pi\#(D_4^-)$
(respectively, $2\pi(\#(D_{4+}^-)-\#(D_{4-}^-))$).
The theorem then follows.
\end{proof}

\appendix
\section{Behavior of the geodesic and normal curvatures}
\label{sec:kgkn}
Let $f:(\R^2,0)\to(\R^3,0)$ be a frontal, and let
$\nu:(\R^2,0)\to\R^3$ be its unit normal vector field.
Let $\gamma:(\R,0)\to(\R^2,0)$ be a curve, and
$\hat\gamma=f\circ\gamma$.
We say that $\gamma$ or $\hat\gamma$ is a {\it frontal as a curve on 
a frontal} if there exist a function $l$ and a non-zero vector 
valued function $e$ such that 
$\hat\gamma'=le$, where $'=d/dt$ and $t$ is a parameter.
Let $s$ be an arc-length parameter of $\hat\gamma$, namely,
$ds=|\hat\gamma'|dt$.
Let $\kappa_g$ and $\kappa_n$ be the geodesic and
the normal curvatures, namely,
$$
\kappa_g=
\dfrac{\det(\hat\gamma',\hat\gamma'',\nu(\gamma))}{|\hat\gamma'|^3},\quad
\kappa_n=
\dfrac{\hat\gamma''\cdot\nu(\gamma)}{|\hat\gamma'|^2}.
$$
We set
$\kappa_g\,ds$ (respectively, $\kappa_n\,ds$)
and call it the {\it geodesic curvature measure}
(respectively, {\it normal curvature measure}\/).
Under this setting, we show the following claim.
\begin{proposition}
The geodesic curvature measure 
and the normal curvature measure are
bounded.
\end{proposition}
\begin{proof}
We show that the function $\phi_g$ (respectively, $\phi_n$)
is bounded 
when we write $\kappa_g\,ds = \phi_g(t)\,dt$
(respectively, $\kappa_n\,ds = \phi_n(t)\,dt$).
Since $\hat\gamma'=le$, we have
$\hat\gamma''=l'e+le'$.
Substituting these formulas and $ds=|\hat\gamma'|dt=|le|dt$,
we obtain
$$
\phi_g
=
\dfrac{|\det(e,e',\nu(\gamma))|}{|e|^2},\quad
\phi_n
=
\dfrac{|e'\cdot\nu(\gamma)|}{|e|}.
$$
This formula
and
$e\ne0$ show the assertion.
\end{proof}


\medskip
{\footnotesize
\begin{flushright}
\begin{tabular}{l}
Department of Mathematics,\\
Graduate School of Science, \\
Kobe University, \\
Rokkodai 1-1, Nada, Kobe \\
657-8501, Japan\\
E-mail: {\tt saji@math.kobe-u.ac.jp}
\end{tabular}
\end{flushright}}

\end{document}